\documentclass[12pt]{article}

\usepackage{amsmath,amsthm}
\usepackage{amssymb}
\usepackage{amsfonts}
\usepackage{amscd}
\usepackage{booktabs}
\input{xy}
\xyoption{all}

\newtheorem{thm}{Theorem}[section]
\newtheorem{theorem}[thm]{Theorem}

\newtheorem{lemma}[thm]{Lemma}

\newtheorem{proposition}[thm]{Proposition}
\theoremstyle{definition}
\newtheorem{definition}[thm]{Definition}

\newtheorem{remark}[thm]{Remark}

\newtheorem{observation}[thm]{Observation}
\newtheorem{definition-proposition}[thm]{Definition-Proposition}

\newcommand{\C} {\mathbf{C}}

\newcommand{\bP}{\mathbb{P}}

\newcommand{\cB}{\mathcal{B}}

\newcommand{\cI}{\mathcal {I}}

\newcommand{\cO}{\mathcal{O}}

\newcommand{\al}{\alpha}
\newcommand{\be}{\beta}
\newcommand{\ga}{\gamma}

\newcommand{\ep}{\epsilon}

\newcommand{\fm}{\mathfrak m}

\newcommand{\salg} {\mathrm{ (salg) }}
\newcommand{\sets} {\mathrm{ (sets) }}

\newcommand\sschemes{ \mathrm{(sschemes)} }

\newcommand{\Gr}{\mathrm{Gr}}

\newcommand{\GL}{\mathrm{GL}}
\newcommand{\rGL}{\mathrm{GL}}

\newcommand{\rEnd}{{\mathrm{End}}}

\newcommand{\rk}{{\mathrm{rk}}}

\newcommand{\pr}{\mathrm{pr}}

\newcommand{\uspec}{\mathrm{\underline{Spec}}}

\newcommand\Spec{{\mathrm{Spec}}}

\newcommand{\Stab}{{\mathrm{Stab}}}

\newcommand{\fbar}{\overline{f}}
\newcommand{\varphibar}{\overline{\varphi}}

\usepackage{color}

\newcommand{\lra} {\longrightarrow}

\newcommand{\beq} {\begin{equation}}
\newcommand{\eeq} {\end{equation}}

\begin{document}

\bigskip

\Large
\centerline{\bf Quotients of complex algebraic supergroups}
\normalsize
\bigskip

\centerline{R. Fioresi$^\flat$, S. D. Kwok$^\star$, D. W. Taylor$^\natural$}

\medskip
\centerline{\it $^\flat$ Dipartimento di Matematica, Universit\`{a} di
Bologna }
 \centerline{\it Piazza di Porta S. Donato, 5. 40126 Bologna. Italy.}
\centerline{{\footnotesize e-mail: rita.fioresi@UniBo.it}}

\medskip
\centerline{\it $^\star$ Department of Biostatistics}
\centerline{\it UCLA, Los Angeles, CA 90095-1772, USA}
\centerline{{\footnotesize e-mail: sdkwok2@gmail.com}}

\medskip
\centerline{\it $^\star$ Department of Mathematics}
\centerline{\it UCLA, 
  CA 90095-1555, USA}
\centerline{{\footnotesize e-mail: dwtaylor@math.ucla.edu}}

\bigskip
\centerline{\textit{\large To the memory of V. S. Varadarajan}}
\bigskip

\begin{abstract}
In this paper we prove that the etale sheafification
of the functor arising from the quotient
of an algebraic supergroup by a closed subsupergroup is representable
by a smooth superscheme.
\end{abstract}

\section{Introduction}
\label{intro-sec}

The purpose of this paper is to provide a construction of
the quotient of a complex algebraic supergroup by a
closed subsupergroup.
This construction is already available in a more
general setting in the literature (see
\cite{mz1}), however here we present a different and more geometric
proof, that is closer to the original approach by Chevalley (see \cite{bo} Ch. II).

\medskip
We start by reviewing the ordinary construction. Suppose $G$ is a complex
algebraic group and $H$ a closed subgroup. 
Then, $G/H$ admits
a unique algebraic variety structure, compatible
with the group multiplication. 
In fact, there exists a rational representation of $G$ in a finite dimensional
vector space $V$ and a line $L$ in $V$ whose stabilizer is $H$.
Hence, we have an action of $G$ on the projective space $\bP(V)$
and $H$ is the stabilizer subgroup of the point $[L]$ in $\bP(V)$.
We can thus identify set-theoretically the quotient $G/H$ with the orbit
$Y$ of the point $[L]$; $Y$ being an orbit is also an algebraic variety,
because of Chevalley's theorem.
The uniqueness of this structure is obtained by the universal property
of the quotient (see \cite{bo} Ch. II).

\medskip
We want to replicate this geometric construction in the super setting. 
There are two major obstructions: the $\C$-points of a supervariety
do not carry enough information on its geometry, 
as it happens for the ordinary counterpart.
Also, quotients of supergroups may not admit a projective embedding.
We overcome the first difficulty by making use of the functor of points
of superschemes and introducing etale coverings
and etale sections, which mimic in some sense
the differential approach to the construction of quotients (see 
\cite{flv, bcf}).
As for the latter problem, we replace the projective superspace
with Grassmannian superschemes. In supergeometry the projective superspace
appears somehow too rigid and it is necessary to allow for more
general structures, as the Grassmannians. In this way 
we can realize an embedding of an orbit of a supergroup
action into a suitable Grassmannian, hence identifying it with a
smooth superscheme. In this sense,
our proof will also provide a variation of the ordinary construction of
quotients of complex algebraic groups 
and goes beyond a mere translation of the known recipe into the
super context.

\medskip
Our main result is the following.

\begin{theorem} \label{mainthm}
Let $G$ be a complex algebraic supergroup, $H$ 
a closed subsupergroup.
Then, the sheafification in the etale topology 
of the functor $T \lra G(T)/H(T)$,
$T$ a superscheme, is representable
in the category of superschemes, by a smooth superscheme.
\end{theorem}

We shall prove this result in several steps.
In Sec. \ref{cat-sec} we give some preliminaries and notation on algebraic
supervarieties and superschemes, while in Sec. \ref{smooth-sec}
we establish some results on smoothness.
In Sec. \ref{etale-sec} 
we prove the representability of the etale sheafification of the
functor $T \lra G(T)/H(T)$, when $H$ is the stabilizer of a point
for an action of $G$ on a superscheme.
Finally, in Sec. \ref{quot-sec} we give our main result, Thm \ref{main}
and a comparison with \cite{mz1} and the definition
by Brundan in \cite{brundan}.

\bigskip
{\bf Acknoledgements}. We are indebted with prof. V.S. Varadarajan
for many illuminating discussions, for all the help and encouragement
given to us through the many years of mutual
interactions. We also thank Prof. D. Gieseker, Prof. A. Maffei
and Prof. T. Graber
for helpful comments. R.F and S.D.K. wish to thank
the UCLA Department of Mathematics for the kind hospitality while
this work was done.

\section{Supervarieties
and Superschemes} \label{cat-sec}

In this section we collect some facts of supergeometry.
For more details see \cite{dm, ma, ccf, vsv2}.

\medskip
Let $\C$ be our ground field. 
Let $\salg$ be the category  of commutative superalgebras
and let $A=A_0 \oplus A_1 \in \salg$. 
Let us consider a non-zero 
$f \in A_0$ and $A_f$ the localization of
the $A_0$-module $A$ at $f$.
The assignment:
\beq \label{sheaf-assign}
U_f:=\{x \in |\Spec(A_0)| \, | \, f(x) \neq 0\} \lra A_f
\eeq
defines a $\cB$-sheaf on $|\Spec(A_0)|$.
Hence, there exists a unique sheaf 
of superalgebras $\cO_A$ on $|\Spec(A_0)|$ such that
$\cO_A|_{U_{f}}=A_f$. 

\begin{definition}
We define 
\textit{affine superscheme} $X$ associated with $A$
the pair $X=(|X|, \cO_A)$, where $|X|$ is the spectrum $|\Spec(A_0)|$ of the
ordinary algebra $A_0$, 
while $\cO_A$ is the sheaf described above.
The \textit{reduced superscheme} $X_r$ underlying $X$ is the ordinary
scheme associated with $A_r=A/J_A$, where $J_A$ is the ideal
of the odd elements in $A$. 

We shall also denote with $\cO_X$ the
sheaf of the superscheme $X$ and with $\cO(X)$ the superalgebra $A$.

\smallskip
A \textit{morphism} $f:X \lra Y$ 
of affine superschemes is a pair $(|f|,f^*)$, where
$|f|:|X|\rightarrow |Y|$ is a continuous map and 
$f^{\ast}:\mathcal{O}_Y\rightarrow f_{\ast}\mathcal{O}_X$ is a map 
of sheaves of superalgebras, such that
$f^{\ast}_p:\mathcal{O}_{Y,|f|(p)}\rightarrow \mathcal{O}_{X,p}$ 
is a local morphism for all $p$ in $|X|$.


\smallskip
We define \textit{superscheme} a pair $X=(|X|,\cO_X)$ consisting
of a topological space $|X|$ and a sheaf of superalgebras $\cO_X$, 
which is locally isomorphic to an affine superscheme.   
\end{definition}

\begin{definition}\label{subvar-def}
Let $X$ be
an affine superscheme, $\cO(X)$ the corresponding superalgebra.
We say that $S$ is a \textit{subscheme} of $X$, if $S$ is the affine
superscheme corresponding to the superalgebra $\cO(X)/I$ for $I$ ideal
in $\cO(X)$. 
If $X=(|X|,\cO_X)$ is a superscheme, 
we say $S=(|S|, \cO_S)$ is a  \textit{subscheme} of $X$, if $|S|$ is
a closed subspace of $|X|$ and $\cO_S=\cO_X/\cI$, where $\cI$ is an ideal sheaf
with the following property. For any affine cover $\{U_i\}$ of $X$, $\cI(U_i)$
is an ideal in $\cO_X(U_i)$, $\cO_S(U_i)=\cO_X(U_i)/\cI(U_i)$.
 and on such
$U_i$ the sheaf $\cO_S|_{U_i}$ is obtained starting from
the superalgebra $\cO_S(U_i)$ as in (\ref{sheaf-assign}).

\end{definition}
We now come to the functor of points.

\begin{definition} \label{Tpt}
\index{$T$-point}
Let $S$ and $T$ be superschemes.  
A \textit{$T$-point} of $S$ is 
a morphism $T \longrightarrow S$.  We denote the set of all $T$-points 
by $S(T)$.  
We define the \textit{functor of points} of the superscheme $S$ 
as the functor:
$$
S: \text{(sschemes)}^o \lra \sets, \quad T \mapsto S(T), \quad
S(\phi)(f)=f \circ \phi,
$$
where $\text{(sschemes)}$ denotes the category of superschemes,
$\sets$ the category of sets and
the index $o$ as usual refers to the opposite category.

\end{definition}
By a common abuse of notation the superscheme $S$ and the
functor of points of $S$ are denoted with the same letter;
whenever is necessary to make a distinction, we shall write
$h_S$ for the functor of points of $S$. 

\begin{definition}\label{affinesuperalg-def}
Let $A$ be a commutative superalgebra, $J_A$ the ideal generated by
the odd elements. We say that $A$ is 
an \textit{affine superalgebra}, if
$A_0$ is a finitely generated superalgebra, 
such that its reduced associated
algebra $A_{r}=A/J_A$ 
is an affine algebra (i.e. finitely generated and with no nilpotents) and
$A_1$ is a finitely  generated $A_0$-module. 

We say that $X$ is an \textit{affine supervariety}, if $X=(|\uspec A|,\cO_A)$
and $A_r$ is an integral domain, i.e. $X_r$ is
an ordinary affine variety. A \textit{supervariety} $X$ is a superscheme
which is locally isomorphic to an affine supervariety.
\end{definition}

\begin{remark}
We are also interested in the functor of points of algebraic
supervarieties, which
are a subcategory of the category of 
superschemes. The category of affine superschemes
is equivalent to the category of commutative superalgebras
(see \cite{ccf} Ch. 10), moreover the functor of points of a superscheme
is determined by its behaviour on affine superschemes. 
We can then regard the functor of points of
an algebraic supervariety (or superscheme) $X$ as starting from the
category of commutative superalgebras, that is
$X: \salg \lra \sets$, $X(\phi)(f)=\phi \circ f$.
\end{remark}


\section{Smooth morphisms}\label{smooth-sec}

We now introduce the notion of {\emph smooth morphism of 
relative dimension}. For the ordinary setting see \cite{mo} Ch. 5.

\begin{definition}\label{smooth-etale-def}
We say that a morphism of superschemes
$f:X \lra Y$ is \textit{smooth} at $x \in |X|$
of relative dimension $m|n$, if there exists two affine neighbourhoods
$U \subset X$ and $V \subset Y$ such that:
$$
\xymatrix{
U \ar[d] \ar@{^{(}->}[r] & \uspec R[x_1 \dots x_{m+r},\xi_1,\dots,\xi_{n+s}]/
(f_1,\dots, f_r,\phi_1,\dots, \phi_s) \ar[d]\\
V      \ar@{^{(}->}[r]     & \uspec R}
$$
and the rank of the Jacobian is maximal, i.e. 
$$
\rk\frac{\partial(f_i,\phi_j)}{\partial(x_k,\xi_l)}(x)=r|s
$$
($1 \leq i \leq r$, $1 \leq j \leq s$, $1 \leq k \leq m+r$, $1 \leq l \leq n+s$).
$f$ is \textit{smooth} of relative dimension $m|n$, 
if it is smooth of relative dimension $m|n$ at all $x \in |X|$.

We say that a morphism of superschemes is \textit{etale},
if it is smooth of relative dimension $0|0$.

We say that $x \in |X|$ is a \textit{smooth point}, if
the corresponding morphism $X \lra \C$ is smooth ($|X|$ is
identified with $X(\C)$, see \cite{ccf} Ch. 10, 10.6.4). The superscheme 
$X$ is \textit{smooth}, if all $x \in |X|$ are smooth. 
\end{definition} 

This notion of smoothness of a superscheme $X$
is equivalent to the one 
in \cite{fi1} and \cite{mz1, mz2}.

\begin{proposition}\label{smooth-prop}
A
morphism of superschemes $f:X \lra Y$ is smooth of relative
dimension $m|n$ at $x \in |X|$ if and only if  
there
exist an open $V\subset Y$, $U=f^{-1}(V)\subset X$  
 ($x \in |U|$) such that $f=\pi \circ g$,
$$
U \stackrel{g}\lra V \times \C^{m|n} \stackrel{\pi}\lra V
$$
where $\pi$ is the projection and $g$ is etale.
\end{proposition}

\begin{proof} 
One direction is clear, since the
composition of smooth morphisms is smooth. 
Since the question is local, we can look at superalgebra maps, that is
$f^*: R$ $\lra$  $R[x_1, \dots, x_m,\dots x_{m+r},\xi_1, \dots, \xi_n, \dots
\xi_{n+s}]/(f_1, \dots, f_r,\phi_1, \dots, \phi_s)$. 
We can write:
$$
\begin{array}{l}
\begin{array}{lr}
R \stackrel{\pi^*}\lra R[x_1, \dots, x_m,\xi_1, \dots, \xi_n] \lra
& \quad \quad \quad
\end{array}\\
\begin{array}{lr}
\quad \quad \quad & 
\stackrel{g^*}\lra R[x_1, \dots, x_m,\dots x_{m+r},\xi_1, \dots, \xi_n, \dots
\xi_{n+s}]/(f_1, \dots, f_r,\phi_1, \dots, \phi_s) 
\end{array}
\end{array}
$$
with
$$
\rk\frac{\partial(f_i,\phi_j)}{\partial(x_k,\xi_l)}(x)=r|s
$$
by the very definition of $g$ etale, the result follows immediately.
\end{proof}

\begin{lemma}\label{etale-lem}
Let $f:X \lra Y$ be a smooth morphism of superschemes
of relative dimension $m|n$.
Then, for any morphism $Y' \lra Y$ we have
that $\pr_2: X \times_Y Y' \lra Y$ is smooth of relative
dimension $m|n$.
$$
\xymatrix{
X \times_Y Y' \ar[d]^{\pr_1} \ar[r]_{\pr_2}
& Y' \ar[d]
\\
 X \ar[r]   & Y
}
$$
In particular,
if $f:X \lra Y$ is etale, also $\pr_2: X \times_Y Y' \lra Y$  is
etale.
\end{lemma}
\begin{proof} 
Since the question is local, we can assume to be in the affine case.
$$
\xymatrix{
  R[x_i,\xi_j]/(f_k,\phi_l) \otimes_R S & S \ar[l] \\ 
  R[x_i,\xi_j]/(f_k,\phi_l) \ar[u]_{\pr_1^*} &\ar[l]\ar[u]  R
}
$$
Since $R[x_i,\xi_j]/(f_k,\phi_l) \otimes_R S \cong S[x_i,\xi_j]/(f_k,\phi_l)$
we obtain the result.
\end{proof}


We now make some observations on Grothendieck topologies. For
more details see \cite{fga} for the ordinary setting and \cite{fz}
for the supergeometric one.

\begin{observation}\label{sites}
Let us consider the category $\sschemes$ of superschemes and 
define coverings of a superscheme $U$ to be collections of etale maps 
whose images cover $U$. 
This is a Grothendieck topology, 
because of the existence and the properties of the 
fibered product in $\sschemes$, together with Lemma \ref{etale-lem}. 
This topology defines the \textit{super Etale site}.
Similarly, we can define another Grothendieck topology 
by taking Zariski coverings, i.e. collections of open
embeddings, and obtain the \textit{super Zariski site} (see
\cite{fz}).  Notice that  
if $U_i \lra U$ is a Zariski covering of a superscheme, then it is 
also an etale covering. Hence the etale topology is finer than the Zariski one.
By the previous observation, we immediately have that a sheaf
on the etale topology is a sheaf in the Zariski one, but not vice-versa
(see \cite{fz} Sec. 2, Prop. 2.5).
\end{observation} 

As in the ordinary setting, any etale morphism will admit
an {\sl etale section}; this fact is 
essential for our construction of quotients.

\begin{proposition}\label{etale-sec-prop}
Let $f:X \lra Y$ be a morphism of superschemes smooth
of relative dimension $m|n$ at $p \in |X|$. Then
there exist open $V \subset Y$, $U \subset X$, $p \in |U|$,
an etale cover $\phi:W \lra V$ and a morphism $W \lra U$ 
making the following diagram commute:
$$
\begin{array}{ccc}
U &\subset &X \\
\uparrow & & \downarrow \\ 
W & \stackrel{\phi}\lra & V \subset Y
\end{array}
$$
\end{proposition}

\begin{proof}
By Prop. \ref{smooth-prop}, we have that there exists $U$ open in $X$
and $V$ open in $Y$ such that  
$$
U \lra V \times \C^{m|n} \lra V
$$
We can write immediately a section $s$ for the projection,
$s:V \lra V \times \C^{m|n}$. 
$$
\xymatrix{
W =U \times_{ V \times \C^{m|n}} V \ar[d] \ar[r] & V \ar[d]_s\\
U \ar[r]^g & V \times \C^{m|n}}
$$
By Lemma \ref{etale-lem}, since $g$ is etale, we have that 
$\pr_2:W=U \times_{ V \times \C^{m|n}} V \lra V$ is also etale.
\end{proof}

The morphism $W \lra U$ is called a {\it local etale section} of
$f: X\lra Y$.


\begin{proposition}\label{keyprop}
Let $f:X \lra Y$ be a morphism of smooth superschemes of finite type,
$X$ an algebraic variety.
If $|f|$ is surjective, 
$(df)_x:T_xX \lra T_{f(x)}Y$ is surjective and
$\dim T_xX - \dim T_{f(x)}Y=m|n$ for all $x \in |X|$, then 
$f$ is smooth of relative dimension $m|n$. 
\end{proposition}

\begin{proof} 
The proof follows closely Cor. 5.4.6, Ch. V in \cite{mo}.
We briefly recap here the main steps. The statement is
local, so let $x \in |X|$.  We can factor $f$ as:
$$
\xymatrix{
U \ar[rd]_f \ar@{^{(}->}[r]
&  Y \times \C^{m|n} \ar[d]^p\\
&Y}
$$
where $U\subset X$ is open, $x \in |U|$. In terms of superalgebra
maps this diagram reads:
$$
\xymatrix{
R[x_1,\dots, x_{m+r},\xi_1,\dots,\xi_{n+s}]/(f_i,\varphi_j)  &  
R[x_1,\dots, x_{m+r},\xi_1,\dots,\xi_{n+s}] \ar@{->>}[l]\\
& R \ar[ul]^{f^*} \ar[u]^{p^*}}
$$
with $i=1,\dots, r$, $j=1,\dots, s$. Furthermore
$(f_i,\varphi_j)$ can be chosen such that
$$
\rk \frac{\partial(f_i,\varphi_j)}{\partial(x_k,\xi_l)}=r|s
$$  
This is because we can choose such $f_i,\varphi_j$ so that
their images $\fbar_i,\varphibar_j$ in $\fm_{X,x}/\fm_{X,x}^2$
are linearly independent.

Since $(df)_x:T_xX \lra T_{f(x)}Y$ is surjective, we have
an embedding \break$\fm_{Y,f(x)}/\fm_{Y,f(x)}^2 \subset \fm_{X,x}/\fm_{X,x}^2$.
Using elementary facts of linear algebra, we have that
$\fbar_i,\varphibar_j$ are independent also in
$\fm_{Z,x}/(\fm_{Z,x}^2+\fm_{X,x}\cO_{x,Z})$ for $Z=Y \times \C^{m|n}$.
This latter condition gives the independence of the 
differentials $df_i,d\varphi_j$, hence the result.
\end{proof}

\section{Etale sections and quotients} \label{etale-sec}

In this section we examine supergroup actions and homogeneous superspaces.

\begin{definition}
A {\it supergroup functor} 
is a group valued functor from $\salg$ to $\sets$.
An {\it affine supergroup} is a supervariety whose functor of points is group
valued, that is to say, it associates a group to each superalgebra.
\end{definition}

If $G$ is an affine supergroup, then $G$ is
a closed subgroup of $\rGL(m|n)$ and the superalgebra $\cO(G)$ has a natural
Hopf superalgebra structure (see \cite{ccf} Ch. 11).
Furthermore, $G$ is smooth (see \cite{fi1}).

\begin{definition} Let $V$ be a super vector space.
We define \textit{ linear representation} 
of $G$ in $V$ a morphism $\rho:  G  \lra \rEnd(V)$
where $\rEnd(V)$ are the endomorphism of $V$.
We will also say that $ G $ \textit{ acts} on $V$.

\smallskip
Let $Y$ be a superscheme. We say that $G$ \textit{acts} on $Y$ if 
we have a morphism of superschemes:
$a: G \times Y  \lra  Y$, 
$g,x  \mapsto  a_T(g,x):=g \cdot x$, $x \in Y(T)$,
$g \in G(T)$,
such that: \\
1. $1 \cdot x=x$, $\forall x \in Y(T)$ \\
2. $(g_1g_2) \cdot x=g_1 \cdot (g_2 \cdot x)$, $\forall x \in Y(T)$,
$\forall g_1, g_2 \in G(T)$.

\smallskip\noindent
For $p \in |Y|$ we define the 
\textit{orbit map} $a_p:G \lra Y$ by $a_{p,T}(g)=g \cdot p$
$\forall g \in G(T)$.

\smallskip\noindent
Let $Y$ be smooth.
We say that the action $a$ is \textit{transitive}, 
if there exists a $p \in |Y|$ 
such that $|a_p|$ and $(da_p)_{1_G}$ are surjective. In this case we call
$Y$ an \textit{homogeneous superspace}.
\end{definition}

Notice that according to Prop. \ref{keyprop}, this is equivalent to ask
that $a_p$ is smooth of relative dimension $m|n=\mathrm{ker}(da_p)_{1_G}$.

\begin{proposition}\label{etale-cov-prop}
Let $G$ be an affine supergroup acting transitively on a smooth superscheme $Y$.
Then there exists an etale cover $\{W_i \lra Y\}$ making the following
diagram commute:
$$
\begin{array}{ccc}
U_i &\subset & G \\
\uparrow & & \downarrow \\ 
W_i & \stackrel{\phi_i}\lra &  Y
\end{array}
$$
where the $U_i$ are open and cover $G$.
\end{proposition}

\begin{proof} Immediate from Prop. \ref{etale-sec-prop}.
\end{proof}

\begin{lemma}\label{key-lem}
Let the notation be as above. 
Let $\al \in Y(Z)$, $Z \in \sschemes$.
Then there exists an etale covering $\{\phi_i:Z_i \lra Z\}$ 
and elements $\be_i \in G(Z_i)$ such that
$$
\begin{array}{ccc}
G(Z_i) & \lra & Y(Z_i)\\
\be_i& \mapsto& \al_i=\al \circ \phi_i
\end{array}
$$
\end{lemma}


\begin{proof}
Let $f_i:W_i \lra V_i$ be the etale covering
of $Y$ described  in Prop. \ref{etale-cov-prop} and
$\sigma_i:W_i \lra U_i \subset G$ the corresponding etale
sections. Let $\al:Z \lra Y$
be a $Z$-point of $Y$. Define $Z_i:=W_i \times_{Y} Z$.
We have the diagram:
$$
\begin{array}{ccc}
Z_i=W_i \times_{Y} Z  &\stackrel{g_i}\lra& Z \\
\pr_1\downarrow & & \downarrow  \\
W_i & \stackrel{f_i}\lra & Y
\end{array}
$$
Since the $f_i$ are etale, we have that the $g_i$ are etale.
Take $\be_i:=\sigma_i \circ \pr_1:Z_i \lra G$; by
the very construction $a_{p,Z_i}(\be_i)=\al_i$.  
\end{proof}

Let $H$ be the stabilizer functor of $p \in |Y|$, that is
$H(Z):=\{g \in G(Z) | g \cdot p=p\}$.
This is representable by a closed subgroup of $G$ (see \cite{ccf} Ch. 11).
We can define the functor:
$$
G/H:\sschemes^o \lra \sets, \qquad (G/H)(Z)=G(Z)/H(Z)
$$
the definition on the arrows being clear.

The morphism $a_{p}$ induces a
natural transformation $G/H \lra Y$, 
with $G(Z)/H(Z) \lra  Y(Z)$ injective for all $Z$.
In general, it will not be surjective, however we have
the following (see \cite{fga, fz} for the notion of sheafification
in this context).

\begin{theorem}\label{sheafif-thm}
Let the notation be as above.
The sheafification $\widetilde{G/H}$ in the etale topology
of the functor $Z \lra G(Z)/H(Z)$ is isomorphic to $Y$ and
it is the functor of points of a superscheme.
\end{theorem}

\begin{proof}
By our previous observation, we have a natural transformation:
$\psi:G/H \lra Y$,
that factors as:
$$
G/H \lra \widetilde{G/H} \stackrel{\psi}\lra Y
$$
We want to show that $\psi$ is an isomorphism. 
We only need to show it is
surjective. Let $\al \in Y(Z)$. Then by Lemma \ref{key-lem}
there exists an etale cover $\phi_i:Z_i \lra Z$ 
and elements $\be_i \in G(Z_i)$ such that
$$
\begin{array}{cccc}
a_{p,Z_i}:&G(Z_i) & \lra & Y(Z_i)\\
&\be_i& \mapsto& \al_i=\al \circ \phi_i
\end{array}
$$
Let $\be_i'$ be the projections of the $\be_i$ onto $G(Z_i)/H(Z_i)$.
We have the commutative diagram
$$
\begin{array}{ccc}
Z_i \times_{G/H} Z_j &\lra & Z_j \\
\downarrow & & \downarrow \be_j'\\
Z_i & \stackrel{\be_i'}\lra &G/H
\end{array}
$$
Hence, the $\be_i'$ correspond to a unique $\be \in \widetilde{G/H}(Z)$, 
so this shows that $\widetilde{G/H}(Z) \cong Y(Z)$.
\end{proof}

\section{Quotients}\label{quot-sec}

In this section we prove our main result.

\begin{proposition} \label{fixedsubspace}
Let the $G$ be an affine
algebraic supergroup and $H$ a closed subsupergroup. 
Then, there exists a finite dimensional representation
$\rho$ of $G$ in $V$ and a subspace $W \subset V$, such that:
$$
H(T)=\{g \in G(T) | \rho(g)W=W\}, \, \, \,
$$
\end{proposition}

\begin{proof} 
  See \cite{ccf}, 11.7.11.
\end{proof}

Once we
fix suitable coordinates, the subsuperspace $W \subset V$  corresponds
to a point $p \in |\Gr|$, where $\Gr$ is the Grassmannian of $r|s$ subsuperspaces of $\C^{m|n}$,
where $r|s=\mathrm{dim}\, W$ and $m|n=\mathrm{dim}\,V$ (see \cite{ccf}
Ch. 10 for the definition of $\Gr$ as superscheme). So we have
an action $a=\rho|_G:G \times \Gr \lra \Gr$, where $H=\Stab \, p$, 
and the corresponding orbit map
$a_p:G \lra \Gr$, $a_{p,T}(g)=g \cdot p$,
for all $g \in G(T)$. 
Notice that both $G$ and $\Gr$ are smooth algebraic varieties;
$a_p$ is of finite type.
By Chevalley's theorem $|a_p|(|G|)$ is open in its closure, hence
it defines a superscheme that we denote by $G \cdot p$ and call the
\textit{orbit} of $p$.
We have then the following commutative diagram:
\beq\label{diagr-G}
\begin{array}{ccc}
\rGL(m|n) & \stackrel{\rho_p}\lra & \Gr \\
\uparrow  &  & \uparrow \\
G & \stackrel{a_p}\lra &G \cdot p 
\end{array}
\eeq
the vertical arrows being injections.

Without loss of generality, choose $p \in \Gr$ as the
subsuperspace $\langle e_1, \dots, e_r$, $\ep_{n-s}, \dots \ep_n\rangle$.
So its stabilizer in $\rGL(m|n)$ is:
$$
P(R)=\left\{\begin{pmatrix} a_{11} & a_{12} & \al_{13} & \al_{14} \\
0 & a_{22} & \al_{23} & 0 \\
0 & \al_{32} & a_{33} & 0 \\
\al_{41} & 0 & 0 & a_{44} \end{pmatrix}\right\} \subset \rGL(m|n)(R)
$$
where $a_{11}$, $a_{44}$ are  $r \times r$, $s \times s$ matrices
with entries in $R_0$, while $\al_{14}$ is $r \times s$,
$\al_{41}$ is $s \times r$ matrix with entries in $R_1$
(similarly for the others).

\begin{observation}\label{nil-obs}
Let
$$
g=\begin{pmatrix} g_{11} & g_{12} & \gamma_{13} & \gamma_{14} \\ 
g_{21} & g_{22} & \ga_{23} & \ga_{34}\\
\ga_{31} & \ga_{32} & g_{33} & g_{34} \\
\ga_{41} & \ga_{42} & g_{34} & g_{44}
\end{pmatrix} \in \GL(m|n)(R)
$$
with $g_{11}$ and $g_{44}$ invertible.

In the
equivalence class $gP(R) \in G(R)/H(R)$, we can choose a unique representative of the form:
$$
\begin{pmatrix} I_r & 0 & 0 & 0 \\ 
u & I_{m-r} & 0 & \eta \\
\xi & 0 & I_{n-s} & v \\
0 & 0 & 0 & I_s\end{pmatrix}
$$
where $I_t$ denotes the identity matrix of rank $t$.

This is a straighforward calculation coming from the fact that the system:
$$
\begin{pmatrix} g_{11} & g_{12} & \gamma_{13} & \gamma_{14} \\ 
g_{21} & g_{22} & \ga_{23} & \ga_{34}\\
\ga_{31} & \ga_{32} & g_{33} & g_{34} \\
\ga_{41} & \ga_{42} & g_{34} & g_{44} \end{pmatrix}=
\begin{pmatrix} I_r & 0 & 0 & 0 \\ 
u & I_{m-r} & 0 & \eta \\
\xi & 0 & I_{n-s} & v \\
0 & 0 & 0 & I_s\end{pmatrix}
\begin{pmatrix} a_{11} & a_{12} & \al_{13} & \al_{14} \\
0 & a_{22} & \al_{23} & 0 \\
0 & \al_{32} & a_{33} & 0 \\
\al_{41} & 0 & 0 & a_{44} \end{pmatrix}
$$
has a unique solution. It is given by:
$$\begin{array}{c}
  a_{11}=g_{11}, \quad a_{12}=g_{12}, \quad \al_{13}=\ga_{13}, \quad \al_{14}=\ga_{14}, \quad \al_{41}=\ga_{41}, \quad a_{44}=g_{44} \\ \\
  a_{22}=g_{22}-ug_{12}, \quad  \al_{23}=\ga_{23}-u\ga_{13},\quad \al_{32}=\ga_{32}-\xi g_{12}, \quad a_{33}=g_{33}-\xi\ga_{13} \\ \\
  \eta=(\ga_{24}-u\ga_{14})g_{44}^{-1}, \quad \eta=(\ga_{31}-u\ga_{41})g_{11}^{-1}, \\ \\
  u=(g_{21}-\ga_{24}g_{44}^{-1}\ga_{41})(g_{11}-\ga_{14}g_{44}^{-1}\ga_{41})^{-1}, \\ \\
  v=(g_{34}-\ga_{31}g_{11}^{-1}\ga_{14})(g_{44}-\ga_{41}g_{11}^{-1}\ga_{14})^{-1}
\end{array}
$$
\end{observation}

\begin{lemma}
The superscheme $G\cdot p$ is smooth.
\end{lemma}

\begin{proof} It is enough to prove smoothness at $p$.
Let $N$ be the closed subsupergroup  of $\rGL(m|n)$ defined
via functor of points by:
$$
N(R)=\left\{ \begin{pmatrix} I_r & 0 & 0 & 0 \\ 
x & I_{m-r} & 0 & \nu \\
\mu & 0 & I_{n-s} & y \\
0 & 0 & 0 & I_s
\end{pmatrix}\right\}
$$
Let $|U|$ be the open subset in $|\GL(m|n)|$ defined by
the open condition $g_{11}$ and $g_{44}$ invertible (see Obs \ref{nil-obs})
and $\pi(|U|)$ its projection on $|\Gr|$, $\pi:|\GL(m|n)| \lra |\Gr|=
|\GL(m|n)|/|P|$. Since $|U|$ and $\pi(|U|)$ are open respectively
in $|\GL(m|n)|$ and $|\Gr|$, they define superschemes, that we denote
with $U$ and $\pi(U)$.

Then, we have a functorial bijection:
$$
\rho_{p,R}: N(R) \lra \pi(U)(R) \subset \Gr(R)
$$
so $N$ and $\pi(U)$ are isomorphic supervarieties.

\medskip
Let $N_G$ be the closed subsupergroup of $G$ defined as
$N_G(R)=N(R) \cap G(R)$. We have $N_G(R)=\pi(U)(R) \cap (G \cdot p)(R)$,
hence $N_G$ is isomorphic to the open
subscheme $\pi(U) \cap G \cdot p$. 
\end{proof}

\begin{proposition} \label{smooth-orbit}
  The action $a:G \times G \cdot p \lra G \cdot p$ is transitive.
\end{proposition}
\begin{proof}
  By the previous lemma, we know that $G \cdot p$ is smooth and
since $G$ is an algebraic supergroup, by \cite{fi1} it is smooth.
It is enough to show that $|a_p|$ is surjective (obvious) and
$(da_p)_{1_G}$ is surjective. By the previous lemma this is clear.
\end{proof}


Now we prove our main result.

\begin{theorem}\label{main}
Let the $G$ be an affine
algebraic supergroup and $H$ a closed subsupergroup. 
The etale sheafification of the functor
\beq\label{fopts-quot}
T \lra G(T)/H(T), \quad T \in \sschemes
\eeq
is representable in the category of superschemes, by a smooth superscheme.
\end{theorem}

\begin{proof}
By Prop. \ref{smooth-orbit}, $G$ acts transitively on the smooth
superscheme $G\cdot p$. Hence, by Prop. \ref{smooth-prop}, $a_p$
is a smooth morphism of relative dimension $m|n$ (for suitable
$m,n$. By Prop. \ref{fixedsubspace}, we have that $H$ is the
stabilizer of a point, so that we can apply Thm. \ref{sheafif-thm} and obtain
the result.  
\end{proof}

We conclude with a comparison with the results in \cite{mz1} and
the definition in \cite{brundan}.

\begin{observation}
\begin{enumerate}
\item The functor of points of a superscheme is a sheaf
in the following Grothendieck topologies: Zariski,
etale and fppf (see \cite{fz, mz1}).
Thm \ref{main} asserts that the etale sheafification of the functor
(\ref{fopts-quot}) is representable by a superscheme $G/H$, hence also its
fppf sheafification has the same property. This is because
$G/H$ is already a sheaf in the fppf topology and the sheafification
construction is unique up to isomorphism (see \cite{fga}).
\item Our realization of quotients satisfies the properties
(Q1)-(Q3) in \cite{brundan} Sec. 2. Properties (Q1), (Q2) are clear
from our construction. As for property (Q3), notice that, 
in diagram (\ref{diagr-G}), $\rho_p$ is an
affine morphism, and the embedding of $G$ into $\rGL(m|n)$ is 
also an affine morphism. 
The Grassmannian $\Gr$ is covered by affine open subsets $V_i=\pi(U_i)$,
$V_i \cap G \cdot p$ is a closed subscheme of $V_i$ hence affine and
open in  $G \cdot p$. By the commutativity of (\ref{diagr-G}),
$a_p^{-1}(V_i \cap G\cdot p)$ is affine. 
\end{enumerate}
\end{observation}

\end{document}